\documentclass[a4paper]{amsart}

\usepackage{amsthm,amssymb,amsmath,bm}

\usepackage[non-sorted-cites,initials]{amsrefs}

\numberwithin{equation}{section}

\newtheorem{lem}{Lemma}[section]
\newtheorem{thm}[lem]{Theorem}
\newtheorem{prop}[lem]{Proposition}
\newtheorem{cor}[lem]{Corollary}

\theoremstyle{definition}
\newtheorem{defn}[lem]{Definition}
\theoremstyle{remark}

\newtheorem{example}[lem]{Example}

\renewcommand{\Re}{\operatorname{Re}}

\newcommand{\Ric}{\operatorname{Ric}}
\newcommand{\tr}{\operatorname{tr}}

\newcommand{\Scal}{\operatorname{Scal}}
\newcommand{\Res}{\operatorname{Res}}
\newcommand{\dvol}{\operatorname{dvol}}
\newcommand{\pr}{\operatorname{pr}}
\newcommand{\id}{\operatorname{id}}

\AtBeginDocument{%
	\def\MR#1{}
}

\title[$Q$-prime curvature and scattering theory]{$Q$-prime curvature and scattering theory
	on strictly pseudoconvex domains}
\author{Yuya Takeuchi}
\address{Graduate School of Mathematical Sciences, 
The University of Tokyo,
3-8-1 Komaba, Meguro, Tokyo 153-8914 Japan}
\email{ytake@ms.u-tokyo.ac.jp}

\begin{document}

\begin{abstract}
The $Q$-prime curvature is a local invariant of
pseudo-Einstein contact forms on 
integrable strictly pseudoconvex CR manifolds.
The transformation law of the $Q$-prime curvature under scaling
is given in terms of a differential operator,
called the $P$-prime operator,
acting on the space of CR pluriharmonic functions.
In this paper,
we generalize these objects to
the boundaries of asymptotically complex hyperbolic Einstein (ACHE) manifolds,
which are partially integrable, strictly pseudoconvex CR manifolds,
by using the scattering matrix for ACHE manifolds.
In this setting,
the $P$-prime operator is a self-adjoint pseudodifferential operator
acting on smooth functions,
and the $Q$-prime curvature is globally determined by the ACHE manifold
and the choice of a contact form on the boundary.
We prove that the integral of the $Q$-prime curvature
is a conformal primitive of the $Q$-curvature;
in particular, it defines an invariant of ACHE manifolds
whose boundaries admit a contact form with zero $Q$-curvature.
We also apply the generalized $Q$-prime curvature
to compute the renormalized volume of strictly pseudoconvex domains
whose boundaries may not admit pseudo-Einstein structure.
\end{abstract}

\maketitle

\section{Introduction} \label{section:introduction}

Scattering theory on asymptotically hyperbolic manifolds
has been extensively developed by several authors;
see, e.g.,~\cites{Graham-Zworski03,Guillarmou-Qing10,Chang-Gonzalez11,Case-Chang16}
and references cited there.
It has important geometric consequences
especially when the asymptotically hyperbolic metric is asymptotically Einstein
(called the Poincar\'e metric);
the scattering matrix has connection to conformally invariant objects
of the boundary at the infinity
--- this can be regarded as an example of the AdS/CFT correspondence
in theoretical physics.
Graham-Zworski~\cite{Graham-Zworski03} proved that
the residues of the scattering matrix $\mathcal{S}(s)$
of the Poincar\'e metric give conformally invariant differential operators,
which have been known as GJMS operators~\cite{Graham-Jenne-Mason-Sparling92}.
When the boundary has even dimension,
they also described Branson's $Q$-curvature~\cite{Branson95}
as a special value of $\mathcal{S}(s)1$.
The scattering matrix has been also related to
the renormalized volume of asymptotically hyperbolic manifolds
by Fefferman-Graham~\cite{Fefferman-Graham02}
and Yang-Chang-Qing~\cite{Yang-Chang-Qing08}
in the case of odd and even dimensional boundary respectively.

Many of the results explained above have analogies
in complex setting,
where we consider asymptotically complex hyperbolic (ACH) Einstein manifolds,
whose boundary has a partially integrable, strictly pseudoconvex CR structure;
see, e.g.,~\cites{Guillarmou-SaBarreto08,Hislop-Perry-Tang08,Frank-Gonzalez-Monticelli-Tan15,Wang15}.
Such manifolds are defined as a generalization
of bounded strictly pseudoconvex domains in $\mathbb{C}^{m}$
with the complete K\"ahler-Einstein metric.
However,
there are CR invariant objects
that are specific to the complex case:
the $P$-prime operator and $Q$-prime curvature,
recently introduced by
Case-Yang~\cite{Case-Yang13} and Hirachi~\cite{Hirachi14},
are constructed by using pluriharmonic functions
and have no natural conformal analogs.
The aim of this paper is to analyze
these objects by using the scattering matrix for ACH Einstein manifolds.

\enlargethispage{1em}
To describe our results,
we first recall basic setup of scattering theory
in the complex case
by following Hislop-Perry-Tang~\cite{Hislop-Perry-Tang08}.
Let $\Omega$ be a bounded strictly pseudoconvex domain
with smooth boundary $M = \partial \Omega$
in a complex manifold $\mathcal{N}$ of complex dimension $m = n+1$.
(We here assume that $\mathcal{N}$ is a complex manifold
to simplify the exposition;
the description for general ACH manifolds will be given in
Sections~\ref{section:background}
and ~\ref{section:Scattering theory on ACH manifolds}.)
A real smooth function $x$ on $\mathcal{N}$ is called a \emph{defining function} of $\Omega$ if
\begin{equation*} 
	\Omega = \{ x > 0 \}, \qquad dx \neq 0 \qquad \text{on} \ M.
\end{equation*}
For a defining function $x$ of $\Omega$,
the $1$-form
\begin{equation} \label{eq:normalized_def}
	\theta = \frac{\sqrt{-1}}{2} (\partial - \overline{\partial}) x |_{M}
\end{equation}
defines a contact form on $M$.
Conversely, if $\theta$ is a contact form on $M$,
we may find a defining function $x$ of $\Omega$
such that~\eqref{eq:normalized_def} holds;
in this case,
we say that $x$ is \emph{normalized by} $\theta$.
We always assume this normalization in the introduction.

By Fefferman~\cite{Fefferman76},
there exists a Hermitian metric $g_{+}$ on $\Omega$
having the following properties:
\begin{itemize}
\item
The Hermitian form $\omega_{+}$ of $g_{+}$ is written in the form
\begin{equation} \label{eq:Kahler}
	\omega_{+} = -\sqrt{-1} \partial \overline{\partial} \log x + \Pi.
\end{equation}
Here $\Pi$ is a $(1,1)$-form on $\mathcal{N}$
and, near $M$, equal to the curvature
\begin{equation*}
	- \sqrt{-1} (m+1)^{-1} \partial \overline{\partial} \log h
\end{equation*}
of a Hermitian metric $h$ of the canonical bundle $K_{\mathcal{N}}$.
\item
The Ricci form $\rho_{g_{+}}$ of $g_{+}$ satisfies
\begin{equation} \label{eq:Einstein}
	\rho_{g_{+}} + (m+1) \omega_{+} = -\sqrt{-1} \partial \overline{\partial} (x^{m+1} \phi)
\end{equation}
for a smooth function $\phi$ on $\mathcal{N}$.
\end{itemize}
Such an $\omega_{+}$ is unique modulo terms of the form
$\sqrt{-1} \partial \overline{\partial}(x^{m+1} \psi)$ with $\psi$ smooth near $M$.

For domains in $\mathbb{C}^{m}$,
we can take $\Pi = 0$ and the Einstein equation~\eqref{eq:Einstein}
is reduced to the complex Monge-Amp\`ere equation for $x$,
which is the original setup of Fefferman.
Introducing the curvature term $\Pi$,
we can naturally generalize his argument
to domains in arbitrary complex manifolds.
Note also that
the decomposition on the right-hand side of~\eqref{eq:Kahler} has ambiguity.
If we choose another Hermitian metric $\widehat{h} = e^{-(m+1)u} h$,
then we obtain another decomposition of $\omega_{+}$
by taking $\widehat{x} = e^{u+\Upsilon} x$
for any pluriharmonic function $\Upsilon$ near $M$.
In particular,
if we can choose $h$ to be flat near $M$, that is, $\Pi =0$ near $M$,
then we can specify $x$ modulo scaling by $e^{\Upsilon}$
with pluriharmonic $\Upsilon$.
We call such an $x$ a \emph{Fefferman's defining function}
and call $\theta$ (normalizing $x$) a \emph{pseudo-Einstein contact form}.
If $\theta$ is pseudo-Einstein,
then $\widehat{\theta} = e^{\Upsilon} \theta$ is pseudo-Einstein
if and only if $\Upsilon$ is a CR pluriharmonic function.

Denote by $\Delta_{+}$ the (non-negative) Laplacian for $2g_{+}$.
Let $m \neq s \in \mathbb{C}$ be in a sufficiently small neighborhood of $m$.
For such $s$ and for $f \in C^{\infty}(M)$,
there exists a unique solution $u$ of the ``Dirichlet problem''
\begin{gather*}
\begin{cases}
	(\Delta_{+} - s(m-s)) u = 0, \\
	u = F x^{m-s} + G x^{s}, \quad
	\text{where} \ F, G \in C^{\infty}(\overline{\Omega}), \\
	F|_{M} = f.
\end{cases}
\end{gather*}
The \emph{Poisson operator}
\begin{equation*} 
	\mathcal{P}_{\theta}(s) \colon C^{\infty}(M) \to C^{\infty}(\Omega)
\end{equation*}
is then defined by $\mathcal{P}_{\theta}(s)f = u$, and the \emph{scattering matrix}
\begin{equation*}
	\mathcal{S}_{\theta}(s) \colon C^{\infty}(M) \to C^{\infty}(M)
\end{equation*}
by $\mathcal{S}_{\theta}(s)f = G|_{M}$.
The operator $\mathcal{S}_{\theta}(s)$ extends meromorphically to $\{ \Re s > m/2 \}$
and, in particular, has a single pole at $s=m$.
In~\cite{Hislop-Perry-Tang08},
Hislop-Perry-Tang proved that
\begin{equation*} 
	P_{\theta}f = -c_{m}^{-1} \Res_{s=m} \mathcal{S}_{\theta}(s)f
\end{equation*}
is a CR invariant differential operator,
known as the \emph{critical GJMS operator}~\cite{Gover-Graham05}.
Here $c_{m} = (-1)^{m}[2m!(m-1)!]^{-1}$.
They also showed that
$\mathcal{S}_{\theta}(s)1$ is holomorphic at $s=m$
and
\begin{equation*} 
	Q_{\theta} = c_{m}^{-1} \lim_{s \to m} \mathcal{S}_{\theta}(s) 1
\end{equation*}
is a local invariant of $\theta$,
which agrees with the \emph{$Q$-curvature};
the integral of this, the \emph{total $Q$-curvature},
defines a global CR invariant~\cite{Fefferman-Hirachi03}.
In fact, the above results can be generalized to ACH manifolds.
We will review the definition and basic properties
of ACH manifolds in Section~\ref{section:background},
and the scattering theory on ACH manifolds in
Section~\ref{section:Scattering theory on ACH manifolds}.

However,
it turns out that the $Q$-curvature vanishes
for a pseudo-Einstein contact form.
Recently,
Case-Yang~\cite{Case-Yang13} and Hirachi~\cite{Hirachi14}
defined the \emph{$P$-prime operator} and \emph{$Q$-prime curvature}.
The $P$-prime operator $P'_{\theta}$ is a differential operator
from the space of CR pluriharmonic functions to $C^{\infty}(M)$
and satisfies $P'_{\theta}1 = Q_{\theta}$ (Definition~\ref{def:p_prime}).
The $Q$-prime curvature $Q'_{\theta} \in C^{\infty}(M)$ is defined
for a pseudo-Einstein contact form $\theta$
and a local invariant of $\theta$  (Definition~\ref{def:q_prime}).
If $\widehat{\theta} = e^{\Upsilon} \theta$ is another pseudo-Einstein contact form,
then one has
\begin{equation*}
	\int_{M} Q'_{\widehat{\theta}}
	= \int_{M} Q'_{\theta}
	+ 2 \int_{M} P'_{\theta} \Upsilon.
\end{equation*}
Here and in the following, for a function $A_{\theta}$ determined by $\theta$,
we denote by $\int_{M} A_{\theta}$ the integral
$\int_{M} A_{\theta} \, \theta \wedge (d\theta)^{n}$.
Since $P'_{\theta}1 = Q_{\theta} =0$ holds,
the \emph{total $Q$-prime curvature}
\begin{equation*}
	\overline{Q'_{\theta}} = \int_{M} Q'_{\theta}
\end{equation*}
is independent of the choice of a pseudo-Einstein contact form
and defines a global CR invariant of $M$
if $P'_{\theta}$ is formally self-adjoint.
For $3$-dimensional CR manifolds,
Case-Yang~\cite{Case-Yang13} showed that $P'_{\theta}$ is formally self-adjoint
by using an explicit formula for $P'_{\theta}$.
On the other hand,
Hirachi~\cite{Hirachi14} proved the formal self-adjointness of $P'_{\theta}$
in the case that $g_{+}$ is K\"ahler
and an additional condition on $\Omega$ is satisfied.
(See the second statement of Theorem~\ref{thm:coincidence_p_prime}.)

We here relate the $P$-prime operator and $Q$-prime curvature
to the scattering theory with the following motivation:
(1) clarify the reason why
we need to impose some global assumptions
for the invariance of the total $Q$-prime curvature
though the $Q$-prime curvature is determined locally by the boundary;
(2) define the $P$-prime operator
as an operator on $C^{\infty}(M)$
and the $Q$-prime curvature for any contact form.

We define a pseudodifferential operator on $C^{\infty}(M)$,
\emph{scattering $P$-prime operator},
and a function on $M$, \emph{scattering $Q$-prime curvature},
by using the scattering theory.
For $f \in C^{\infty}(M)$ and a contact form $\theta$,
the scattering $P$-prime operator $\bm{P}'_{\theta}$ is defined by
\begin{equation*}
	c_{m}^{-1}\mathcal{S}_{\theta}(s)f
	= -\frac{1}{s-m} P_{\theta}f + \bm{P}'_{\theta} f + O(s-m).
\end{equation*}
Similarly, the scattering $Q$-prime curvature $\bm{Q}'_{\theta}$ is given by
\begin{equation*}
	c_{m}^{-1} \mathcal{S}_{\theta}(s)1
	= Q_{\theta} - \frac{1}{2} (s-m) \bm{Q}'_{\theta} + O((s-m)^{2}).
\end{equation*}
These definitions can be generalized to ACH manifolds
(Definitions~\ref{def:scat_p_prime} and~\ref{def:scat_q_prime}).
Analytic properties of $\bm{P}'_{\theta}$ and $\bm{Q}'_{\theta}$
follow easily from the definitions.

\begin{prop} \label{prop:scat_p_prime}
The operator $\bm{P}'_{\theta}$ is formally self-adjoint,
and if $\widehat{\theta} = e^{\Upsilon} \theta$ with $\Upsilon \in C^{\infty}(M)$ is another contact form,
then
\begin{equation} \label{eq:trans_law_scat_pprime}
	e^{m\Upsilon} \bm{P}'_{\widehat{\theta}} f = \bm{P}'_{\theta} f + \Upsilon P_{\theta} f + P_{\theta}(\Upsilon f).
\end{equation}
\end{prop}

Note that, as an operator on $C^{\infty}(M)$,
the scattering $P$-prime operator is different from
the $P$-prime operator on $3$-dimensional
CR manifolds given in~\cite{Case-Yang13}.

\begin{prop} \label{prop:scat_q_prime}
Under the scaling $\widehat{\theta} = e^{\Upsilon} \theta$,
one has
\begin{equation} \label{eq:trans_law_scat_qprime}
	e^{m \Upsilon} (\bm{Q}_{\widehat{\theta}} - 2 \Upsilon Q_{\widehat{\theta}})
	= \bm{Q}'_{\theta} + 2 \bm{P}'_{\theta} \Upsilon + P_{\theta} \Upsilon^{2},
\end{equation}
and
\begin{equation} \label{eq:trans_law_total_scat_qprime}
	\int_{M} \bm{Q}'_{\widehat{\theta}}
	= \int_{M} \bm{Q}'_{\theta} + 2\int_{M} \Upsilon Q _{\theta} + 2\int_{M} \Upsilon Q_{\widehat{\theta}}.
\end{equation}
In particular,
$\int_{M} \bm{Q}'_{\theta}$ is independent of the choice of $\theta$
satisfying $Q_{\theta} = 0$.
\end{prop}
From~\eqref{eq:trans_law_total_scat_qprime},
we can also see that the integral of the scattering $Q$-prime curvature
is a conformal primitive of the $Q$-curvature.

\begin{cor}
If $\theta(\varepsilon) = e^{\varepsilon \Upsilon} \theta$, then
\begin{equation*}
	\left. \frac{d}{d\varepsilon} \right|_{\varepsilon = 0} \int_{M} \bm{Q}'_{\theta(\varepsilon)}
	= 4 \int_{M} \Upsilon Q_{\theta}.
\end{equation*}
\end{cor}

Now we examine the relation between
the $P$-prime operator and $Q$-prime curvature,
and the scattering ones.

We first consider the $P$-prime operator.
Assume that $f$ is a CR pluriharmonic function on $M$.
By the strict pseudoconvexity,
$f$ can be extended
to a smooth function $\widetilde{f}$ on $\overline{\Omega}$
that is pluriharmonic near $M$.
We can write the difference between $P'_{\theta}f$ and $\bm{P}'_{\theta} f$ by using $\widetilde{f}$.

\begin{thm} \label{thm:coincidence_p_prime}
Let $\Omega$ be a bounded strictly pseudoconvex domain
with smooth boundary $M= \partial \Omega$
in an $m$-dimensional complex manifold.
Assume that $\Omega$ is equipped with a Hermitian metric $g_{+}$
satisfying~\eqref{eq:Kahler} and~\eqref{eq:Einstein}.
Let $f$ be a CR pluriharmonic function on $M$ and
$\widetilde{f}$ a smooth extension of $f$ that is pluriharmonic near $M$.
Then 
\begin{equation*}
	K \widetilde{f} = c_{m}^{-1} x^{-m} \mathcal{R}(m) (\Delta_{+} \tilde{f})
\end{equation*}
is smooth up to the boundary and satisfies
\begin{equation*}
	\bm{P}'_{\theta}f = P'_{\theta} f - (K \widetilde{f}) |_{M}.
\end{equation*}
Here $\mathcal{R}(m)$ is the inverse of $\Delta_{+}$
as an operator on the space of square-integrable functions on $(\Omega, g_{+})$.
In particular,
the $P$-prime operator agrees with the scattering $P$-prime operator
on the space of CR pluriharmonic functions
and is formally self-adjoint
if $g_{+}$ is K\"ahler
and any CR pluriharmonic function on $M$
has pluriharmonic extension to $\Omega$.
\end{thm}

Theorem~\ref{thm:coincidence_p_prime} suggests that
the formal self-adjointness of the $P$-prime operator stems from that of the scattering one.

We next consider the $Q$-prime curvature.

\begin{thm} \label{thm:coincidence_q_prime}
Let $(\Omega, g_{+})$ be as in Theorem~\ref{thm:coincidence_p_prime}.
Assume that $\theta$ is a pseudo-Einstein contact form on $M$
and $x$ is a Fefferman's defining function normalized by $\theta$.
Then
\begin{equation*}
	D = - c_{m}^{-1} x^{-m} \mathcal{R}(m) (\tr_{g_{+}} \Pi)
\end{equation*}
is a well-defined smooth function on $\overline{\Omega}$
and satisfies
\begin{equation*}
	\bm{Q}'_{\theta} = Q'_{\theta} - 2D|_{M}.
\end{equation*}
Here the function $\tr_{g_{+}} \gamma$ for a $(1,1)$-form $\gamma$
is defined by
\begin{equation*}
	(\tr_{g_{+}}\gamma) \cdot \omega_{+}^{m} = m \gamma \wedge \omega_{+}^{n}.
\end{equation*}
In particular,
the $Q$-prime curvature coincides with the scattering one
if one can choose $x$ so that $\Pi = 0$.
\end{thm}

We next give an expression for the renormalized volume
in terms of the scattering $Q$-prime curvature and a correction term.
The volume of $\Omega_{\varepsilon} = \{z \in \Omega \mid x(z) > \varepsilon \}$
admits an expansion, as $\varepsilon \to +0$,
\begin{equation*}
	\int_{\Omega_{\varepsilon}} \dvol_{g_{+}}
	= \sum_{j=1}^{m} \frac{b_{j}}{\varepsilon^{m+1-j}} + L \log \varepsilon + V + o(1),
\end{equation*}
where $b_{j}, L, V \in \mathbb{R}$.
It is known that
$L$ is a constant multiple of the total $Q$-curvature~\cite{Seshadri07}.
We call $V$ the \emph{renormalized volume of $(\Omega, g_{+})$ with respect to $x$}.
Hirachi-Marugame-Matsumoto~\cite{Hirachi-Marugame-Matsumoto17}
proved the formula
\begin{equation} \label{eq:renormalized_volume2}
	V = \frac{1}{n!} c_{m} \int_{M} Q'_{\theta}
	+ \frac{1}{m!} \int_{\Omega} \Pi^{m}
\end{equation}
for a pseudo-Einstein contact form $\theta$ and a Fefferman's defining function $x$.
Note that $\Pi$ is a compactly supported $(1,1)$-form on $\Omega$,
and the integral $\int_{\Omega} \Pi^{m}$ is well-defined.
We deduce a formula of the renormalized volume
for any $(\Omega, g_{+})$ in a complex manifold.
For any contact form $\theta$,
there exists a defining function $x$ normalized by $\theta$
such that
\begin{equation*}
	\tr_{g_{+}} \Pi = m - \Delta_{+} \log x = O(x^{m}),
\end{equation*}
and such an $x$ is unique modulo $O(x^{m+1})$;
see Proposition~\ref{prop:flat_def_function}.

\begin{thm} \label{thm:renormalized_volume}
Let $\Omega$ be a bounded strictly pseudoconvex domain
with smooth boundary $M= \partial \Omega$
in a complex manifold of dimension $m = n+1$.
Assume that $\Omega$ is equipped with a Hermitian metric $g_{+}$
satisfying~\eqref{eq:Kahler} and~\eqref{eq:Einstein}.
Moreover, assume that $x$ is chosen
so that $\tr_{g_{+}} \Pi = O(x^{m})$.
Then the function $- (d/ds)|_{s=m} \mathcal{P}_{\theta}(s)1$ is of the form
\begin{equation} \label{eq:correction_term}
	- \left. \frac{d}{ds} \right|_{s=m} \mathcal{P}_{\theta}(s)1 = \log x + Ax^{m} + Bx^{m} \log x,
\end{equation}
where $A, B \in C^{\infty}(\overline{\Omega})$,
and the renormalized volume $V$ with respect to $x$ is written in the form
\begin{equation} \label{eq:renormalized_volume}
	V = \frac{1}{n!} c_{m} \int_{M} \bm{Q}'_{\theta}
	- \frac{1}{n!} \int_{M} A|_{M} \, \theta \wedge (d\theta)^{n}.
\end{equation}
\end{thm}

Note that both terms on the right-hand side in~\eqref{eq:renormalized_volume}
depend on the choice of $g_{+}$
and the second term also depends on the choice of $x$.
We can decompose these terms
into the total $Q$-prime curvature and a contribution from the interior
if $\theta$ is pseudo-Einstein.

\begin{thm} \label{thm:renormalized_volume2}
Let $(\Omega, g_{+})$ be as in Theorem~\ref{thm:renormalized_volume}.
Assume that $g_{+}$ is K\"ahler
and $x$ is a Fefferman's defining function.
Then
\begin{equation} \label{eq:integral_scat_qprime}
	\int_{M} \bm{Q}'_{\theta} = \int_{M} Q'_{\theta}
	+ \frac{2}{mc_{m}} \int_{\Omega} \Pi^{m}.
\end{equation}
Moreover,
the function $B$ in~\eqref{eq:correction_term} is equal to zero
and
\begin{equation} \label{eq:integral_correction}
	\int_{M} A|_{M} \, \theta \wedge (d\theta)^{n}
	= \frac{1}{m} \int_{\Omega} \Pi^{m}.
\end{equation}
In particular,
one has the formula~\eqref{eq:renormalized_volume2}.
\end{thm}

For $4$-dimensional ACHE manifolds,
Herzlich~\cite{Herzlich07} proved a formula for the renormalized volume
in terms of the Euler characteristic,
the integral of a Riemannian invariant of ACHE manifolds,
and the integral of a pseudo-Hermitian invariant of the boundary.
On the other hand,
Seshadri~\cite{Seshadri07} deduced another formula for the renormalized volume
of $2$-dimensional K\"ahler-Einstein manifolds
with strictly pseudoconvex boundary.
However,
their choices of defining functions are different from ours.
Note also that
in the case of asymptotically hyperbolic manifolds,
Fefferman-Graham~\cite{Fefferman-Graham02}
and Yang-Chang-Qing~\cite{Yang-Chang-Qing08}
considered the function corresponding to $\bm{Q}'_{\theta}$
for the study of the renormalized volume.

This paper is organized as follows.
In Section~\ref{section:background},
we review basic notions of partially integrable CR manifolds
and ACH metrics.
In Section~\ref{section:Scattering theory on ACH manifolds},
we recall the scattering theory on ACH manifolds,
define the scattering $P$-prime operator and $Q$-prime curvature,
and prove propositions.
In Section~\ref{section:Proof of Theorems},
we prove theorems stated in this section.

\section{Asymptotically complex hyperbolic manifolds} \label{section:background}

\subsection{Partially integrable CR manifolds} \label{section:partially_integrable}

Let $M$ be a smooth manifold of dimension $2n+1$, $n \geq 1$.
An \emph{almost CR structure} is a complex $n$-dimensional subbundle $T^{1,0}M$
of the complexified tangent bundle $T^{\mathbb{C}}M$
such that $T^{1,0}M \cap T^{0,1}M = 0$, where $T^{0,1}M = \overline{T^{1,0}M}$.
We say that an almost CR structure $T^{1,0}M$ is
\emph{partially integrable} if the following condition is satisfied:
\begin{equation*}
	[ \Gamma(T^{1,0}M) , \Gamma(T^{1,0}M) ] \subset \Gamma(T^{1,0}M \oplus T^{0,1}M). 
\end{equation*}
In particular
if $[ \Gamma(T^{1,0}M) , \Gamma(T^{1,0}M) ]$ is contained in $\Gamma(T^{1,0}M)$,
the almost CR structure is said to be \emph{integrable}.
For example, if $M$ is a real hypersurface in a complex manifold $\mathcal{N}$,
then $M$ has a canonical integrable CR structure
\begin{equation*}
	T^{1,0}M = T^{1,0} \mathcal{N} |_{M} \cap T^{\mathbb{C}}M.
\end{equation*}
A smooth function $f$ is said to be a \emph{CR function}
if $df$ annihilates $T^{0,1}M$.
A \emph{CR pluriharmonic function} is a real smooth function
that is locally the real part of a CR function.
Assume that
there exists a nowhere vanishing real $1$-form $\theta$
that annihilates $\mathcal{H} = \Re T^{1,0}M$.
The \emph{Levi form} $L_{\theta}$ of $\theta$ is defined by 
\begin{equation*}
	L_{\theta}(Z,\overline{W}) = - \sqrt{-1} d\theta(Z,\overline{W}),
	\qquad Z, W \in \Gamma(T^{1,0}M).
\end{equation*}
In the following, we assume that $T^{1,0}M$ is \emph{strictly pseudoconvex},
i.e., the Levi form is positive definite for some choice of $\theta$.
Such a $\theta$ is called a \emph{pseudo-Hermitian structure}, or a \emph{contact form}.

\subsection{$\Theta$-manifolds} \label{section:th_manifolds}
The intrinsic definitions of $\Theta$-manifolds and ACH metrics
are given in~\cites{Epstein-Melrose-Mendoza91,Guillarmou-SaBarreto08,Matsumoto14}.
Here we give these definitions following~\cite{Matsumoto16}.
Let $X$ be a $(2n+2)$-dimensional smooth manifold with boundary $M$.
Denote by $\iota$ the inclusion $M \hookrightarrow X$.
Let $\Theta \in \Gamma(T^{*}X|_{M})$ be a smooth $1$-form such that
$\theta = \iota^{*}{\Theta}$ is nowhere vanishing.
A \emph{$\Theta$-structure} on $X$
is a conformal class $[\Theta]$,
and a pair $(X, [\Theta])$ is called a \emph{$\Theta$-manifold}.
A \emph{$\Theta$-diffeomorphism} between $\Theta$-manifolds
is a diffeomorphism that preserves the $\Theta$-structures.

For a $\Theta$-manifold $(X, [\Theta])$,
there exists a canonical vector bundle $^{\Theta}TX$,
the \emph{$\Theta$-tangent bundle} on $X$,
which is defined by modifying the tangent bundle near the boundary.
This vector bundle is canonically isomorphic to
the usual tangent bundle over the interior $\mathring{X}$ of $X$.
The structure of $^{\Theta}TX$ near the boundary is as follows.
Let $p \in M$
and $N, T, Y_{1}, \dots , Y_{2n}$ a local frame of $TX$
in a neighborhood $U$ of $p$ such that
\begin{itemize}
	\item $N|_{\partial X}$ is annihilated by $\Theta$;
	\item $T,Y_{1}, \dots , Y_{2n}$ are tangent to $M \cap U$;
	\item $( Y_{1}|_{\partial X}, \dots , Y_{2n}|_{\partial X} )$
		is a local frame of $\ker \theta$ on $M \cap U$.
\end{itemize}
Then the vector bundle $^{\Theta}TX|_{U}$ is spanned by
$(\rho N, \rho^{2} T, \rho Y_{1}, \dots, \rho Y_{2n})$.
Here $\rho \in C^{\infty}(X)$ is a defining function of $X$, that is,
$\rho$ satisfies
\begin{equation*}
	\mathring{X} = \{ \rho > 0 \},
	\qquad d\rho \neq 0 \quad \text{on} \ M.
\end{equation*}
A bundle metric of the $\Theta$-tangent bundle is called a \emph{$\Theta$-metric}.
This induces a complete Riemannian metric on $\mathring{X}$.

Assume that the boundary $M$ has a partially integrable CR structure $T^{1,0}M$.
We say that it is \emph{compatible with the $\Theta$-structure}
if $\theta$ is a contact form on $M$.

\subsection{ACH metrics} \label{section:ach_metrics}
Let $(M, T^{1,0}M)$ be a partially integrable CR manifold of dimension $2n+1$,
and $\theta$ a contact form on $M$.
Denote  by $\pr$ the projection $[ 0, \infty ) \times M \to M$.
Then the $1$-form $\Theta = \pr^{*} \theta$ defines a $\Theta$-structure on $[0,\infty) \times M$
and the conformal class $[\Theta]$ is independent of the choice of $\theta$.
Moreover, the partially integrable CR structure $T^{1,0}M$
is compatible with this $\Theta$-structure.
We call this $\Theta$-structure the \emph{standard $\Theta$-structure}
on $[0, \infty) \times M$. 
Let $(Z_{1}, \dots , Z_{n})$ be a local frame of $T^{1,0}M$ and set
\begin{equation*}
	\bm{Z}_{\infty} = r \frac{\partial}{\partial r}, \quad \bm{Z}_{0} = r^{2} T,
	\quad \bm{Z}_{\alpha} = r Z_{\alpha},
	\quad \bm{Z}_{\overline{\alpha}} = r Z_{\overline{\alpha}},
\end{equation*}
where $r$ is the coordinate of $[0, \infty )$ and $T$ is the Reeb vector field
associated to $\theta$.
Then the complexified $\Theta$-tangent bundle
is spanned by $(\bm{Z}_{I}) = (\bm{Z}_{\infty}, \bm{Z}_{0}, \bm{Z}_{\alpha}, \bm{Z}_{\overline{\alpha}})$. 

Let $U \subset [0,\infty) \times M$ be an open neighborhood of $\{0\} \times M$.
A $\Theta$-metric $g$ on $U$ is said to be an \emph{ACH metric}
if, for some choice of $\theta$, the boundary values of $g_{IJ} = g(\bm{Z}_{I},\bm{Z}_{J})$
satisfy
\begin{align*}
	g_{\infty \infty} &= 4,& g_{\infty 0} &= 0, & g_{\infty \alpha} &= 0, \\
	g_{00} &= 1,& g_{0\alpha} &= 0,& g_{\alpha \overline{\beta}} &= h_{\alpha \overline{\beta}},
	\qquad g_{\alpha \beta} = 0 \quad \text{on} \quad \{0\} \times M.
\end{align*}
Here $h_{\alpha \overline{\beta}} = L_{\theta}(Z_{\alpha}, Z_{\overline{\beta}})$.

When $(X, [\Theta])$ is an arbitrary $\Theta$-manifold,
an ACH metric on $X$ is defined as follows.

\begin{defn}
Let $X$ be a $\Theta$-manifold.
A $\Theta$-metric $g$ on $X$ is called an \emph{ACH metric}
if there exist the following objects:
\begin{itemize}
	\item a compatible partially integrable CR structure $T^{1,0}M$ on $M$;
	\item an open neighborhood $U \subset [0,\infty) \times M$ of $M$,
		an open neighborhood $V \subset X$ of $M$,
		and a $\Theta$-diffeomorphism $\Psi\colon U \to V$
		such that $\Psi|_{M} = \id_{M} $ and
		$\Psi^{*}g$ is an ACH metric on $U$.
\end{itemize}
The partially integrable CR structure $T^{1,0}M$ is determined uniquely by $g$
and called the \emph{CR structure at infinity}.
A compact $\Theta$-manifold with an ACH metric is called
an \emph{ACH manifold}.
\end{defn}

Let $\rho$ be a defining function of a $\Theta$-manifold $X$ with an ACH metric $g$.
The symmetric $2$-tensor $\rho^{4} g$ on $\mathring{X}$
is smooth up to the boundary
and $\iota^{*}(\rho^{4}g)(Y,\cdot) = 0$ holds for $Y \in \mathcal{H}$.
Hence there exists a unique contact form $\theta$ on $M$
such that $\iota^{*}(\rho^{4} g) = \theta \otimes \theta$;
in this case, we say that a defining function $\rho$ is \emph{normalized by $\theta$}. 

For a given $\Theta$-manifold $X$
and a compatible partially integrable CR structure $T^{1,0}M$,
one can construct an ACH metric on $X$ with CR structure at infinity $T^{1,0}M$
satisfying
\begin{equation*}
	\Ric_{g} + \frac{n+2}{2} g \in \rho^{2n+2} \Gamma(S^{2}(^{\Theta}T^{*}X))
\end{equation*}
and
\begin{equation*}
	\Scal_{g} + (n+1)(n+2) \in \rho^{2n+3} C^{\infty}(X),
\end{equation*}
where $\Ric_{g}$ and $\Scal_{g}$ are respectively
the Ricci tensor and the scalar curvature of $g$
as a Riemannian metric on $\mathring{X}$~\cite{Matsumoto16}*{Theorem 2.5}.
We call this metric an \emph{asymptotically complex hyperbolic Einstein metric},
an \emph{ACHE metric} for short.
A compact $\Theta$-manifold with an ACHE metric
is called an \emph{ACHE manifold}.

\begin{example} \label{example}
Let $\Omega$ be a strictly pseudoconvex domain in a complex manifold $\mathcal{N}$
and $g_{+}$ a Hermitian metric satisfying~\eqref{eq:Kahler} and~\eqref{eq:Einstein}.
We define an ACHE manifold $X$ from $(\Omega,g_{+})$ as follows.
As a manifold,
$X$ is $\overline{\Omega}$ with $C^{\infty}$ structure replaced by
adding the square roots of defining functions of $\Omega$.
The $\Theta$-structure on $X$ is given by the pullback of a $1$-form
$(\sqrt{-1}/2) (\partial - \overline{\partial}) x$ by the identity map $ i \colon X \to \overline{\Omega}$,
where $x$ is a defining function of $\Omega$.
The $\Theta$-metric $g$ is defined by $i^{*} (2g_{+})$.
Then $X$ is an ACHE manifold, 
and the CR structure at infinity is the same as
the integrable CR structure induced from the complex structure on $\mathcal{N}$.
Let $\rho = i^{*}(\sqrt{x/2})$ be a defining function of $X$.
Then for a contact form $\theta$ on $M$,
$\rho$ is normalized by $\theta$ if and only if $x$ is normalized by $\theta$.
\end{example}

\subsection{GJMS operator and $Q$-curvature} \label{section:gjms_op}

In this subsection,
we review the definitions of the critical GJMS operator and $Q$-curvature
for partially integrable CR manifolds.

Let $g$ be an ACH metric on a $\Theta$-manifold $X$ of dimension $2m = 2(n+1)$,
and $T^{1,0}M$ its CR structure at infinity.
Fix a contact form $\theta$ on $M$
and a defining function $\rho$ of $X$ normalized by $\theta$,
and set $x = 2\rho^{2}$.
Let us denote by $\Delta$
the Laplacian of $g$. 
To simplify the notation,
we use the symbol $O(\cdot)$ as follows.
For $a \in \mathbb{C}$ and $l \in \mathbb{N} = \{ 0,1,2, \dotsc \}$,
a function of the form $x^{a} \sum_{j=0}^{l} f_{j} (\log x)^{j}$
with $f_{j} \in C^{\infty}(X)$
is represented by $O(x^{a} (\log x)^{l})$.
The symbol $O(x^{\infty})$ denotes
a smooth function on $X$
that vanishes to infinite order at the boundary.
From the expression of $\Delta$ given in~\cite{Matsumoto16}*{Proposition 3.1},
we have the following

\begin{lem} \label{lem:laplacian}
Let $f$ be a smooth function defined near the boundary of $X$.
For $s \in \mathbb{C}$ and $j \in \mathbb{N}$,
\begin{align}
    \label{eq:laplacian0}
	&\quad\ (\Delta - s(m-s))(f \cdot x^{m-s + j/2})  \\
\notag	& = - \frac{1}{4}j (2m-4s + j)f \cdot x^{m-s+j/2} + O(x^{m-s+(j+1)/2}),\\[1ex]
    \label{eq:laplacian1}
	&\quad\ (\Delta -s(m-s))(f \cdot x^{m-s+j/2} \log x)  \\
	\notag & = - \frac{1}{4}j (2m-4s + j)f \cdot x^{m-s+j/2} \log x \\
	\notag & \quad - (m -2s + j) f \cdot x^{m-s+ j/2}
	+O(x^{m-s+(j+1)/2} \log x),  
\end{align}
and
\begin{align*}
	&\quad\ (\Delta - s(m-s))(f \cdot x^{m-s+j/2} (\log x)^{2}) \\
	& = - \frac{1}{4}j (2m-4s + j)f \cdot x^{m-s+j/2}(\log x)^{2} \\
	& \quad - 2(m-2s + j) f \cdot x^{m-s+j/2} \log x  -2 f \cdot x^{m-s+j/2} \\
	& \quad + O(x^{m-s+(j+1)/2} (\log x)^{2} ).
\end{align*}
\end{lem}

This Lemma gives an important tool
to define various objects.
The first example is the following

\begin{thm}[{\cite{Matsumoto16}*{Theorem 3.3}}] \label{thm:formal_p}
Let $\theta$ be a contact form on $M$
and $\rho$ a defining function of $X$ normalized by $\theta$.
For any real valued smooth function $f$ on $M$,
there exists a solution $u \in C^{\infty}(\mathring{X})$ of the equation
\begin{equation} \label{eq:char_p_eq}
	\Delta u = O(x^{\infty})
\end{equation}
with the expansion
\begin{equation} \label{eq:exp_p}
	u = F + G x^{m} \log x
\end{equation}
for $F,G \in C^{\infty}(X)$ such that $F|_{M} = f$.
The function $F$ is unique modulo $O(x^{m})$,
and $G$ is unique modulo $O(x^{\infty})$.
Moreover, there is a differential operator
\begin{equation*}
	P_{\theta}\colon C^{\infty}(M) \to C^{\infty}(M)
\end{equation*}
such that
\begin{equation*}
	G|_{M} = -2c_{m} P_{\theta}f.
\end{equation*}
The operator $P_{\theta}$ is formally self-adjoint
and annihilates constant functions.
One has $e^{m\Upsilon} P_{\widehat{\theta}} = P_{\theta}$
under the conformal change $\widehat{\theta} = e^{\Upsilon} \theta$.
\end{thm}

Here we give the existence part of the proof
since the same argument will be used repeatedly. 

\begin{proof}
Let $f \in C^{\infty}(M)$.
First take a smooth extension $f_{0}$ of $f$.
Let $k < 2m-1$ and assume that
$f_{0}, \dots , f_{k} \in C^{\infty}(X)$ satisfy
\begin{equation*}
	\Delta \left( \sum_{j=0}^{k} f_{j}x^{j/2} \right) \in x^{(k+1)/2}C^{\infty}(X).
\end{equation*}
If we take
\begin{equation*}
	f_{k+1} = \frac{4}{(k+1) (-2m + k+1)} x^{-(k+1)/2} 
	\Delta \left( \sum_{j=0}^{k} f_{j}x^{j/2} \right) \in C^{\infty}(X),
\end{equation*}
then~\eqref{eq:laplacian0} gives
\begin{equation*}
	\Delta \left( \sum_{j=0}^{k+1} f_{j} x^{j/2} \right) \in x^{(k+2)/2}C^{\infty}(X).
\end{equation*}
This argument stops at $k=2m-1$ since $-2m + k+1 = 0$.
So we need logarithmic terms.
If we choose
\begin{equation*}
	g_{0} = \frac{1}{m} x^{-m} \Delta \left( \sum_{j=0}^{2m} f_{j} x^{j/2} \right),
\end{equation*}
then~\eqref{eq:laplacian1} gives
\begin{equation*}
	\Delta \left( \sum_{j=0}^{2m} f_{j} x^{j/2} + g_{0} x^{m} \log x \right)
	= O(x^{(2m+1)/2} \log x).
\end{equation*}
This inductive argument gives $F$ and $G$.
\end{proof}

If the metric $g$ is an ACHE metric,
$P_{\theta}$ is determined by $(M,T^{1,0}M)$ and $\theta$.
Then we call $P_{\theta}$ the \emph{critical GJMS operator}.

Next, consider the $Q$-curvature.

\begin{thm}[{\cite{Matsumoto16}*{Theorem 3.5}}] \label{thm:formal_q}
There exists a solution $v \in C^{\infty}(\mathring{X})$ of the equation
\begin{equation}
	\Delta v = m + O(x^{\infty})
\end{equation}
with the expansion
\begin{equation}
	v = \log x + A + B x^{m} \log x
\end{equation}
for $A,B \in C^{\infty}(X)$ such that $A|_{M} = 0$.
The function $A$ is unique modulo $O(x^{m})$,
and $B$ is unique modulo $O(x^{\infty})$.
Moreover, the function $Q_{\theta}$ defined by
\begin{equation*}
	B|_{M} = -2c_{m} Q_{\theta}
\end{equation*}
satisfies the following transformation law
under the conformal change $\widehat{\theta} = e^{\Upsilon} \theta$:
\begin{equation*}
	e^{m \Upsilon} Q_{\widehat{\theta}} = Q_{\theta} + P_{\theta} \Upsilon.
\end{equation*}
\end{thm}

The function $Q_{\theta}$ is also determined by $(M,T^{1,0}M)$ and $\theta$
if the metric $g$ is an ACHE metric;
in this case, we call $Q_{\theta}$ the \emph{$Q$-curvature}.

From Theorem~\ref{thm:formal_q},
we have a normalization of defining functions,
which is a generalization of Fefferman's defining functions. 

\begin{prop}\label{prop:flat_def_function}
For any contact form $\theta$,
there exists a defining function $\widetilde{\rho}$ of $X$ normalized by $\theta$
such that $\widetilde{x} = 2{\widetilde{\rho}}^{2}$ satisfies
\begin{equation} \label{eq:flat_def_function}
	\Delta \log \widetilde{x} = m + O(x^{m}).
\end{equation}
Moreover,
such a $\widetilde{\rho}$ is unique modulo $O(x^{(2m+1)/2})$
and such an $\widetilde{x}$ is unique modulo $O(x^{m+1})$.
\end{prop}

\begin{proof}
Let $v =  \log x + A + B x^{m} \log x$ be as in Theorem~\ref{thm:formal_q}.
Then $\widetilde{\rho} = e^{A/2} \rho$ is also
a defining function of $X$ normalized by $\theta$,
and $\widetilde{x}$ satisfies~\eqref{eq:flat_def_function};
this proves the existence of $\widetilde{\rho}$.
Uniqueness of $\widetilde{\rho}$ (and $\widetilde{x}$) follows from~\eqref{eq:laplacian0}.
\end{proof}

\section{Scattering theory on ACH manifolds} 
\label{section:Scattering theory on ACH manifolds}
In this section,
assume that 
$X$ is an ACH manifold of dimension $2m = 2(n+1)$
and $\rho$ is a defining function of $X$
normalized by a contact form $\theta$ on $M$.
Set $x = 2\rho^{2}$ to simplify the notation.

\subsection{Poisson operator and scattering matrix} \label{section:poisson_op}
First we review the scattering theory on ACH manifolds.
Some basic results about the Laplacian $\Delta$ are obtained by
Epstein-Melrose-Mendoza~\cite{Epstein-Melrose-Mendoza91}.

\begin{thm} \label{thm:resolvent}
(1) The pure point spectrum $\sigma_{pp} (\Delta)$ 
is a finite subset of $(0, m^2/4)$,
and the absolutely continuous spectrum $\sigma_{ac}(\Delta)$ is 
the interval $[m^2/4, \infty)$.

(2) For $s \in \mathbb{C}$ with
$\Re s> m/2$ and $s(m-s) \notin \sigma_{pp}(\Delta)$,
the resolvent
\begin{equation*}
	\mathcal{R}(s)=(\Delta-s(m-s))^{-1} 
\end{equation*}
is an operator from $x^{\infty} C^{\infty}(X)$ to $x^{s} C^{\infty} (X) $
and holomorphic in $s$.
Here $x^{\infty} C^{\infty}(X)$ stands for the space of
smooth functions on $X$ that vanish to infinite order at the boundary.
\end{thm}

Let $s \in \mathbb{C}$ with
\begin{equation*}
	\Re s > m/2, \qquad 4s-2m \notin \mathbb{Z}, \qquad s(m-s) \notin \sigma_{pp}(\Delta).
\end{equation*}
For each $f \in C^{\infty}(M)$,
there exists a unique solution $u \in C^{\infty}(\mathring{X})$ of the equation
\begin{equation} \label{eq:poisson_eq}
	(\Delta - s(m-s))u = 0
\end{equation}
with the expansion
\begin{equation} \label{eq:exp_poisson}
	u = Fx^{m-s} + G x^{s}
\end{equation}
for $F,G \in C^{\infty}(X)$ such that $F|_{M} = f$.
The \emph{Poisson operator}
\begin{equation*}
	\mathcal{P}_{\theta}(s)\colon C^{\infty}(M) \to C^{\infty}(\mathring{X})
\end{equation*}
is then defined by $\mathcal{P}_{\theta}(s)f = u$,
and the \emph{scattering matrix}
\begin{equation*}
	\mathcal{S}_{\theta}(s)\colon C^{\infty}(M) \to C^{\infty}(M)
\end{equation*}
by $\mathcal{S}_{\theta}(s)f = G|_{M}$.
These two operators are holomorphic in $s$ and determined by $\theta$.
The Poisson operator is holomorphic across $s=m$.
The scattering matrix is formally self-adjoint for $s \in \mathbb{R}$
and satisfies
\begin{equation} \label{eq:trans_law_scat_mat}
	e^{s\Upsilon}\mathcal{S}_{\widehat{\theta}}(s) f = \mathcal{S}_{\theta}(s)(e^{(m-s)\Upsilon} f)
\end{equation}
if $\widehat{\theta} = e^{\Upsilon} \theta$.

As we will use the explicit form of these operators near $s=m$,
we here recall the constructions of $\mathcal{P}_{\theta}(s)$ and $\mathcal{S}_{\theta}(s)$.

\begin{lem} 
Let $s \in \mathbb{C}$ and $k \in \mathbb{N}$.
Then there exist differential operators $p_{k,s}$ on $M$
that are determined by the following conditions:
\begin{itemize}
	\item $p_{0,s} = 1$;
	\item the operator $p_{k,s}$ is holomorphic in $s$
		outside the discrete set
		$\Lambda_{k} = \{ s \in \mathbb{C} \mid 4s-2m \in \mathbb{N} \cap [0,k] \}$,
		and has a single pole at each $s_{0} \in \Lambda_{k}$;
	\item for $s \notin \Lambda_{k}$
		and $f \in C^{\infty}(M)$,
		\begin{equation} \label{eq:asymp}
			(\Delta -s(m-s))\left( x^{m-s} \sum_{j=0}^{k} (p_{j,s}f)
			x^{j/2} \right)
			\in x^{m-s+ (k+1)/2} C^{\infty}(X).
		\end{equation}
\end{itemize}
\end{lem}

\begin{proof}
The operator $p_{k,s}$ is determined inductively from~\eqref{eq:laplacian0}
as in the proof of Theorem~\ref{thm:formal_p}.
\end{proof}

Using this Lemma,
we would like to construct a holomorphic family $u_{s}$
in $C^{\infty}(\mathring{X})$ near $s=m$
satisfying
\begin{equation} \label{eq:asymp_poisson_eq}
	\begin{cases}
	(\Delta - s(m-s))u_{s} = O(x^{\infty}), \\
	u_{s} - fx^{m-s} = o(x^{m-s}).
	\end{cases}
\end{equation}
However,
this is not obvious since $p_{k,s}$ has a single pole at $s=m$ for $k \geq 2m$.
To avoid this difficulty,
we need additional arguments.

\begin{prop} \label{prop:formal_solution}
Let $f$ be a smooth function on $M$
and $s \in \mathbb{C}$ with $\Re s > m/2$.
There are meromorphic families $\mathcal{F}_{s}$ and $\mathcal{G}_{s}$
of smooth functions on $X$
satisfying the following:
\begin{enumerate}
	\item $\mathcal{F}_{s}$ and $\mathcal{G}_{s}$ are holomorphic in $s$ near $s = m$;
	\item $\mathcal{F}_{s}|_{M} = (s-m) f$;
	\item $\mathcal{F}_{m} + \mathcal{G}_{m} x^{m} \in x^{\infty}C^{\infty}(X)$;
	\item the function
		\begin{equation*}
			u_{s} =
			\begin{cases}
				(s-m)^{-1} \left[\mathcal{F}_{s} x^{m-s} + \mathcal{G}_{s} x^{s}
				- \mathcal{F}_{m} - \mathcal{G}_{m} x^{m} \right] & s \neq m, \\
				\Dot{\mathcal{F}}_{m} + \Dot{\mathcal{G}}_{m} x^{m}
				+ (- \mathcal{F}_{m} + \mathcal{G}_{m} x^{m}) \log x & s = m,
			\end{cases}
		\end{equation*}
		is a solution of~\eqref{eq:asymp_poisson_eq}.
\end{enumerate}
Here $\Dot{\mathcal{F}}_{m} = (d/ds)|_{s=m} \mathcal{F}_{s}$
and $\Dot{\mathcal{G}}_{m} = (d/ds)|_{s=m} \mathcal{G}_{s}$.
\end{prop}

\begin{proof}
Taking the residue of~\eqref{eq:asymp} at $s=m$,
we have
\begin{equation*}
	\Delta \left(x^{m} \sum_{j=0}^{k} (\Res_{s=m} p_{2m+j,s} f) x^{j/2} \right) 
	\in x^{m+ (k + 1)/2} C^{\infty}(X).
\end{equation*}
While replacing $f$ with $\Res_{s=m} p_{2m,s}f$
and substituting $s=0$ into~\eqref{eq:asymp},
we obtain
\begin{equation*}
	\Delta \left(x^{m} \sum_{j=0}^{k} (p_{j,0} \Res_{s=m}p_{2m,s}f) x^{j/2} \right)
	\in x^{m+(k+1)/2}C^{\infty}(X).
\end{equation*}
From the construction of $p_{k,s}$,
we observe
\begin{equation*}
	\Res_{s=m} p_{2m+j,s} f = p_{j,0}\left(\Res_{s=m} p_{2m,s}f \right).
\end{equation*} 
Consequently,
the function
\begin{equation*}
	x^{m-s} \sum_{j=0}^{2m+k} (s-m) (p_{j,s}f) x^{j/2}
	- x^{s} \sum_{j=0}^{k} p_{j,0}(\Res_{s=m} p_{2m,s}f) x^{j/2}
\end{equation*}
is holomorphic near $s=m$ and equal to zero at $s=m$.
Moreover
the image of this function by $\Delta - s(m-s)$ is contained in
$x^{2m-s+(k +1)/2} C^{\infty}(X) + x^{s + (k + 1)/2} C^{\infty}(X)$.
Borel's lemma gives holomorphic families $\mathcal{F}_{s}$ and $\mathcal{G}_{s}$
in $C^{\infty}(X)$ near $s=m$ satisfying (1)--(3).

Since $\mathcal{F}_{s}$ and $\mathcal{G}_{s}$ are holomorphic,
the function
\begin{equation*}
	u_{s} = \frac{1}{s-m} \left[\mathcal{F}_{s} x^{m-s} + \mathcal{G}_{s} x^{s}
	- \mathcal{F}_{m} - \mathcal{G}_{m} x^{m} \right]
\end{equation*}
extends holomorphically across $s=m$ in $C^{\infty}(\mathring{X})$.
In particular,
\begin{equation*}
	u_{m} = \Dot{\mathcal{F}}_{m} + \Dot{\mathcal{G}}_{m} x^{m}
	+ (- \mathcal{F}_{m} + \mathcal{G}_{m} x^{m}) \log x.
\end{equation*}
Moreover
$(\Delta-s(m-s))u_{s}$ is also holomorphic near $s=m$ in $x^{\infty}C^{\infty}(X)$.
Hence $\mathcal{F}_{s}$ and $\mathcal{G}_{s}$ satisfy (4) also.
\end{proof}

Using Proposition~\ref{prop:formal_solution},
we can write explicit forms of the Poisson operator
and scattering matrix.

\begin{thm} \label{thm:explicit_poisson_op}
For $f \in C^{\infty}(M)$,
let $\mathcal{F}_{s}$, $\mathcal{G}_{s}$ and $u_{s}$ be as in
Proposition~\ref{prop:formal_solution}.
Set
\begin{equation*}
	\mathcal{H}_{s} = - x^{-s} \mathcal{R}(s)(\Delta -s(m-s))u_{s} \in C^{\infty}(X).
\end{equation*}
Then
\begin{equation*}
	\mathcal{P}_{\theta}(s)f =
	\begin{cases}
	(s-m)^{-1} [\mathcal{F}_{s} x^{m-s} + \mathcal{G}_{s} x^{s}
	- \mathcal{F}_{m}- \mathcal{G}_{m} x^{m} ] + \mathcal{H}_{s} x^{s} & s \neq m, \\
	\Dot{\mathcal{F}}_{m} + (\Dot{\mathcal{G}}_{m}+ \mathcal{H}_{m}) x^{m}
	+ (-\mathcal{F}_{m} + \mathcal{G}_{m} x^{m}) \log x & s =m,
	\end{cases}
\end{equation*}
and
\begin{equation*}
	\mathcal{S}_{\theta}(s)f = \frac{1}{s-m} \mathcal{G}_{s}|_{M} + \mathcal{H}_{s}|_{M}.
\end{equation*}
\end{thm}

\begin{proof}
The function $u_{s}$ is holomorphic in $s$ near $s=m$
as a $C^{\infty}(\mathring{X})$-valued function
and satisfies $(\Delta -s(m-s)) u_{s} = O(x^{\infty})$.
Hence the function
\begin{equation*}
	u = [I - \mathcal{R}(s)(\Delta-s(m-s))]u_{s}
\end{equation*}
satisfies the equation~\eqref{eq:poisson_eq}.
Moreover,
$\mathcal{R}(s)(\Delta - s(m-s)) u_{s}$ is in $x^{s}C^{\infty}(X)$ from Theorem~\ref{thm:resolvent}.
Hence $u$ also has the expansion~\eqref{eq:exp_poisson}.
\end{proof}

\begin{cor}
\begin{equation*}
	\Res_{s=m} \mathcal{S}_{\theta}(s)f = - c_{m} P_{\theta}f.
\end{equation*}
\end{cor}

\begin{proof}
If $s=m$,
$\mathcal{P}_{\theta}(m)f$ is of the form
\begin{equation*}
	\Dot{\mathcal{F}}_{m} + 2\mathcal{G}_{m}x^{m} \log x + x^{1/2} C^{\infty}(X)
\end{equation*}
with $\Dot{\mathcal{F}}_{m}|_{M} = f$.
Hence $\mathcal{P}_{\theta}(m)f$ satisfies~\eqref{eq:char_p_eq} and~\eqref{eq:exp_p},
and we have
\begin{equation*}
	-2c_{m}P_{\theta}f
	= 2\mathcal{G}_{m}|_{M}
	= 2 \Res_{s=m} \mathcal{S}_{\theta}(s)f.
\end{equation*}
This proves the corollary.
\end{proof}

Next, we consider the case $f=1$.
In this case, $p_{k,s}1$ is holomorphic near $s=m$ for all $k \in \mathbb{N}$.
Hence Borel's lemma gives the solution of~\eqref{eq:asymp_poisson_eq}.

\begin{prop} \label{prop:formal_solution2}
Let $s \in \mathbb{C}$ with $\Re s > m/2$.
There exists a smooth function $\mathcal{I}_{s}$ on $X$
such that
\begin{enumerate}
	\item $\mathcal{I}_{s}$ is holomorphic near $s=m$;
	\item $\mathcal{I}_{s}|_{M} = 1$;
	\item $u_{s} = \mathcal{I}_{s}x^{m-s}$ satisfies the equation~\eqref{eq:asymp_poisson_eq} for $f=1$.
\end{enumerate}
\end{prop}

The same argument
as in the proof of Theorem~\ref{thm:explicit_poisson_op}
gives the following

\begin{thm} \label{thm:explicit_poisson_op2}
Let $\mathcal{I}_{s}$ be as in Proposition~\ref{prop:formal_solution2},
and set
\begin{equation*}
	\mathcal{J}_{s} = - x^{-s} \mathcal{R}(s)(\Delta - s(m-s))(\mathcal{I}_{s}x^{m-s}).
\end{equation*}
Then
\begin{equation*}
	\mathcal{P}_{\theta}(s)1 = \mathcal{I}_{s} x^{m-s} + \mathcal{J}_{s}x^{s},
\end{equation*}
and
\begin{equation*}
	\mathcal{S}_{\theta}(s)1 = \mathcal{J}_{s}|_{M}.
\end{equation*}
In particular, $\mathcal{S}_{\theta}(s)1$ is holomorphic at $s=m$.
\end{thm}

\subsection{Scattering $P$-prime operator and $Q$-prime curvature}
\label{section:scat_p_prime}

In this subsection, we give the definitions of the scattering $P$-prime operator
and $Q$-prime curvature.

\begin{defn} \label{def:scat_p_prime}
The \emph{scattering $P$-prime operator} $\bm{P}'_{\theta} \colon C^{\infty}(M) \to C^{\infty}(M)$
is defined by
\begin{equation*}
	c_{m}^{-1}\mathcal{S}_{\theta}(s)f = - \frac{1}{s-m} P_{\theta}f + \bm{P}'_{\theta} f + O(s-m).
\end{equation*}
\end{defn}

Analytic properties of $\bm{P}'_{\theta}$ stated in Section~\ref{section:introduction}
follow easily from the definition.

\begin{proof}[Proof of Proposition~\ref{prop:scat_p_prime}]
Since $\mathcal{S}_{\theta}(s)$ is
formally self-adjoint for $s \in \mathbb{R}$,
so is $\bm{P}'_{\theta}$.
The constant term of the Laurent series
of the left-hand side of~\eqref{eq:trans_law_scat_mat} at $s=m$
is equal to 
\begin{equation*}
	- c_{m} \Upsilon e^{m\Upsilon} P_{\widehat{\theta}}f + c_{m} e^{m\Upsilon}\bm{P}'_{\widehat{\theta}}f,
\end{equation*}
while that of the right-hand side of~\eqref{eq:trans_law_scat_mat} at $s=m$
is equal to
\begin{equation*}
	c_{m} P_{\theta}(\Upsilon f) + c_{m} \bm{P}'_{\theta}f.
\end{equation*}
Hence we have~\eqref{eq:trans_law_scat_pprime}.
\end{proof}

We also give a characterization of $\bm{P}'_{\theta}$
by the Laplacian.

\begin{prop} \label{prop:formal_scat_p_prime}
For $f \in C^{\infty}(M)$,
there are $\bm{F}'$, $\bm{G}'$, and $\bm{H}' \in C^{\infty}(X)$
such that $\bm{F}'|_{M} = 0$ and 
\begin{equation} \label{eq:exp_scat_p_prime}
	\bm{u}' = (\mathcal{P}_{\theta}(m)f) \log x + \bm{F}' + \bm{G}' x^{m} \log x
	+ \bm{H}' x^{m} (\log x)^{2}
\end{equation}
satisfies
\begin{equation} \label{eq:char_scat_p_prime}
	\Delta \bm{u}' = m \mathcal{P}_{\theta}(m) f + O(x^{\infty}).
\end{equation}
The function $\bm{F}'$ is unique modulo $O(x^{m})$,
and $\bm{G}'$ and $\bm{H}'$ are unique modulo $O(x^{\infty})$.
Moreover
\begin{equation}
	\bm{G}'|_{M} = -2c_{m} \bm{P}'_{\theta}f, \qquad \bm{H}'|_{M} = 2c_{m}P_{\theta}f.
\end{equation}
\end{prop}

\begin{proof}
We use the notation in Theorem~\ref{thm:explicit_poisson_op}.
Since
\begin{align*}
	\Delta [(\mathcal{P}_{\theta}(m)f) \log x]
	&= m \mathcal{P}_{\theta}(m)f - 4m\mathcal{G}_{m} x^{m} \log x\\
	&\quad + O(x^{1/2}) + O(x^{m+1/2} \log x),
\end{align*}
the same argument as in the proof of Theorem~\ref{thm:formal_p}
gives the first and second statements.
The differentiation of the equation $(\Delta -s(m-s)) \mathcal{P}_{\theta}(s)f = 0$ in $s$ gives
\begin{equation*}
	\Delta (\Dot{\mathcal{P}}_{\theta}(m)f)+ m\mathcal{P}_{\theta}(m)f = 0 ,
\end{equation*}
where $\Dot{\mathcal{P}}_{\theta}(m)f = (d/ds)|_{s=m} \mathcal{P}_{\theta}(s)f$.
This shows that $\bm{u}' =-\Dot{\mathcal{P}}_{\theta}(m)f$ satisfies the equation
$\Delta \bm{u}' = m \mathcal{P}_{\theta}(m)f$.
On the other hand,
\begin{align*}
	\Dot{\mathcal{P}}_{\theta}(m)f &= \left. \frac{1}{2} \frac{d^2}{ds^{2}} \right| _{s=m} 
	[\mathcal{F}_{s} x^{m-s} + \mathcal{G}_{s}x^{s} ]
	+ \left. \frac{d}{ds} \right| _{s=m}(\mathcal{H}_{s} x^{s}) \\
	&= - (\mathcal{P}_{\theta}(m)f) \log x + 2 (\Dot{\mathcal{G}}_{m} + \mathcal{H}_{m}) x^{m} \log x \\
	& \quad + 2 \mathcal{G}_{m}x^{m} (\log x)^{2} + O( x^{1/2} ),
\end{align*}
and $\bm{u}'$ has the expansion~\eqref{eq:exp_scat_p_prime} also.
Hence
\begin{equation*}
	\bm{G}'|_{M} = -2 (\Dot{\mathcal{G}}_{m}|_{M} + \mathcal{H}_{m}|_{M})
	= -2c_{m} \bm{P}'_{\theta}f,
\end{equation*}
and
\begin{equation*}
	\bm{H}'|_{M} = -2\mathcal{G}_{m}|_{M} = 2c_{m}P_{\theta}f,
\end{equation*}
and the proof is complete.
\end{proof}

\begin{cor}
One has $\bm{P}'_{\theta}1 = Q_{\theta}$.
\end{cor}

\begin{proof}
If $f=1$,
then $\mathcal{P}_{\theta}(m)1 = 1$.
Hence the function $v$ in Theorem~\ref{thm:formal_q}
satisfies~\eqref{eq:exp_scat_p_prime} and~\eqref{eq:char_scat_p_prime}.
Therefore we have $-2c_{m} \bm{P}'_{\theta} 1 = -2c_{m} Q_{\theta}$.
\end{proof}

Next, we consider the scattering $Q$-prime curvature.

\begin{defn} \label{def:scat_q_prime}
The \emph{scattering $Q$-prime curvature} $\bm{Q}'_{\theta}$ is defined by
\begin{equation*}
	c_{m}^{-1} \mathcal{S}_{\theta}(s)1 = Q_{\theta} - \frac{1}{2} (s-m) \bm{Q}'_{\theta} + O((s-m)^{2}).
\end{equation*}
\end{defn}

\begin{proof}[Proof of Proposition~\ref{prop:scat_q_prime}]
The coefficient of $(s-m)$ in the Laurent series
of the left-hand side of~\eqref{eq:trans_law_scat_mat} at $s=m$
is equal to 
\begin{equation*}
	c_{m} \Upsilon e^{m\Upsilon} Q_{\widehat{\theta}} - \frac{1}{2} c_{m} e^{m\Upsilon} \bm{Q}'_{\widehat{\theta}},
\end{equation*}
while that of the right-hand side of~\eqref{eq:trans_law_scat_mat} at $s=m$
is equal to
\begin{equation*}
	- \frac{1}{2} c_{m} P_{\theta}(\Upsilon^{2})
	- c_{m} \bm{P}'_{\theta} \Upsilon - \frac{1}{2} c_{m} \bm{Q}'_{\theta}.
\end{equation*}
Hence we have~\eqref{eq:trans_law_scat_qprime}.
The integration of the both sides of~\eqref{eq:trans_law_scat_qprime} gives
\begin{align*}
\int_{M} \bm{Q}'_{\widehat{\theta}} \, \widehat{\theta} \wedge (d \widehat{\theta})^{n} 
& =\int_{M} \bm{Q}'_{\theta} \, \theta \wedge (d\theta)^{n}
	+ \int_{M} \Upsilon^{2} (P_{\theta}1) \, \theta \wedge (d\theta)^{n} \\
	& \quad +2\int_{M} \Upsilon (\bm{P}'_{\theta} 1) \, \theta \wedge (d\theta)^{n}
	+ 2\int_{M} \Upsilon Q_{\widehat{\theta}} \, \widehat{\theta} \wedge (d\widehat{\theta})^{n} \\
	&= \int_{M} \bm{Q}'_{\theta} \, \theta \wedge (d\theta)^{n}
	+ 2 \int_{M} \Upsilon Q_{\theta} \, \theta \wedge (d\theta)^{n}\\
	&\quad + 2 \int_{M} \Upsilon Q_{\widehat{\theta}} \, \widehat{\theta} \wedge (d\widehat{\theta})^{n}.
\end{align*}
The last equality is a consequence of $P_{\theta}1 = 0$ and $\bm{P}'_{\theta}1 = Q_{\theta}$.
\end{proof}

The following proposition is a characterization
of $\bm{Q}'_{\theta}$ as an obstruction to smooth solutions.

\begin{prop} \label{prop:formal_scat_q_prime}
There exist $\bm{A}'$, $\bm{B}'$, and $\bm{C}' \in C^{\infty}(X)$ such that $\bm{A}'|_{M} = 0$ and
\begin{equation}
	\bm{v}' = \bm{A}' + \bm{B}'x^{m} \log x + \bm{C}'x^{m} (\log x)^{2}
\end{equation}
satisfies
\begin{equation*}
	\Delta \bm{v}' = 2 \| d(\Dot{\mathcal{P}}_{\theta}(m)1 ) \|^{2}_{g} -2 + O(x^{\infty}).
\end{equation*}
The function $\bm{A}'$ is unique modulo $O(x^{m})$,
and $\bm{B}'$ and $\bm{C}'$ are unique modulo $O(x^{\infty})$.
Moreover
\begin{equation} \label{eq:coeff_scat_qprime}
	\bm{B}'|_{M} = -2c_{m} \bm{Q}'_{\theta}, \qquad \bm{C}'|_{M} = 4c_{m}Q_{\theta}.
\end{equation}
\end{prop}

\begin{proof}
The first and second statements follow from the same argument
as in the proof of Theorem~\ref{thm:formal_p}.
To prove~\eqref{eq:coeff_scat_qprime},
we use the notation in Theorem~\ref{thm:explicit_poisson_op2}.
The function $\mathcal{P}_{\theta}(s)1$ is of the form $\mathcal{I}_{s} x^{m-s} + \mathcal{J}_{s} x^{s} $.
In particular, $\mathcal{P}_{\theta}(m)1 = \mathcal{I}_{m} + \mathcal{J}_{m} x^{m} = 1$
and $\mathcal{I}_{s}|_{M} = 1$ hold.
Hence
\begin{equation*}
	\Dot{\mathcal{P}}_{\theta}(m)1
	= - \log x + \Dot{\mathcal{I}}_{m} + \Dot{\mathcal{J}}_{m} x^{m}
	+ 2\mathcal{J}_{m}x^{m} \log x,
\end{equation*}
where $\Dot{\mathcal{I}}_{m} = (d/ds)|_{s=m} \mathcal{I}_{s}$ and
$\Dot{\mathcal{J}}_{m} = (d/ds)|_{s=m} \mathcal{J}_{s}$, and
\begin{equation*}
	\Ddot{\mathcal{P}}_{\theta}(m)1
	= (\log x )^{2} - 2\Dot{\mathcal{I}}_{m} \log x
	+ 2\Dot{\mathcal{J}}_{m} x^{m} \log x + O( x^{1/2} ),
\end{equation*}
where $\Ddot{\mathcal{P}}_{\theta}(m)1 = (d^{2}/ds^{2})|_{s=m} \mathcal{P}_{\theta}(s)1$.
Therefore
$\Ddot{\mathcal{P}}_{\theta}(m)1 - (\Dot{\mathcal{P}}_{\theta}(m)1)^{2}$ is of the form
\begin{equation*}
	4\Dot{\mathcal{J}}_{m} x^{m} \log x + 4 \mathcal{J}_{m} x^{m} (\log x)^{2}
	+ O(x^{1/2}) + O(x^{m+1/2} \log x).
\end{equation*}
Note that $\mathcal{J}_{m}|_{M} = c_{m}Q_{\theta}$
and $-2 \Dot{\mathcal{J}}_{m}|_{M} = c_{m}\bm{Q}'_{\theta}$
since $\mathcal{J}_{s}|_{M} = \mathcal{S}_{\theta}(s)1$.
On the other hand, the differentiation of $(\Delta -s(m-s))\mathcal{P}_{\theta}(s)1 = 0$ in $s$ gives
\begin{equation*}
	\Delta (\Dot{\mathcal{P}}_{\theta}(m)1) = -m ,
\end{equation*}
and
\begin{equation*}
	\Delta (\Ddot{\mathcal{P}}_{\theta}(m)1) = -2 -2m\Dot{\mathcal{P}}_{\theta}(m)1. 
\end{equation*}
From these, we obtain
\begin{equation*}
	\Delta \left[ \Ddot{\mathcal{P}}_{\theta}(m)1 - (\Dot{\mathcal{P}}_{\theta}(m)1)^{2} \right]
	= 2 \| d (\Dot{\mathcal{P}}_{\theta}(m)1) \|_{g}^{2} -2.
\end{equation*}
Hence
\begin{equation*}
	\bm{B}'|_{M} = 4 \Dot{\mathcal{J}}_{m}|_{M} = - 2c_{m} \bm{Q}'_{\theta},
\end{equation*}
and
\begin{equation*}
	\bm{C}'|_{M} = 4\mathcal{J}_{m}|_{M} = 4c_{m} Q_{\theta}.
\end{equation*}
This completes the proof.
\end{proof}

\section{Proof of theorems} \label{section:Proof of Theorems}

In this section,
we give proofs of the theorems stated in Section~\ref{section:introduction}.
Let $\Omega$ be a strictly pseudoconvex domain
with smooth boundary $M = \partial \Omega$
in a complex manifold $\mathcal{N}$ of complex dimension $m = n+1$,
and $g_{+}$ a Hermitian metric on $\Omega$
satisfying~\eqref{eq:Kahler} and~\eqref{eq:Einstein}.
Fix a contact form $\theta$ on $M$ and
a defining function $x$ of $\Omega$ normalized by $\theta$.
As stated in Example~\ref{example},
an ACHE manifold $X$ is constructed from $(\Omega, g_{+})$.
Then we can translate the scattering theory on $X$ considered in
Section~\ref{section:Scattering theory on ACH manifolds}
into the scattering theory on $\Omega$
as in Section~\ref{section:introduction}.
In particular, we rewrite Propositions~\ref{prop:formal_scat_p_prime}
and~\ref{prop:formal_scat_q_prime}
as the statements for $(\Omega, g_{+})$.

\begin{prop} \label{prop:formal_scat_p_prime2}
For $f \in C^{\infty}(M)$,
there are $\bm{F}'$, $\bm{G}'$, and $\bm{H}' \in C^{\infty}(\overline{\Omega})$
such that $\bm{F}'|_{M} = 0$ and 
\begin{equation*}
	\bm{u}' = (\mathcal{P}_{\theta}(m)f) \log x + \bm{F}' + \bm{G}'x^{m} \log x
	+ \bm{H}' x^{m} (\log x)^{2}
\end{equation*}
satisfies
\begin{equation*}
	\Delta_{+} \bm{u}' = m \mathcal{P}_{\theta}(m) f + O(x^{\infty}).
\end{equation*}
The function $\bm{F}'$ is unique modulo $O(x^{m})$,
and $\bm{G}'$ and $\bm{H}'$ are unique modulo $O(x^{\infty})$.
Moreover
\begin{equation} \label{eq:coeff_scat_pprime2}
	\bm{G}'|_{M} = -2c_{m} \bm{P}'_{\theta}f, \qquad \bm{H}'|_{M} = 2c_{m}P_{\theta}f.
\end{equation}
\end{prop}

\begin{prop} \label{prop:formal_scat_q_prime2}
There exist $\bm{A}'$, $\bm{B}'$, and $\bm{C}' \in C^{\infty}(\overline{\Omega})$
such that $\bm{A}'|_{M} = 0$ and
\begin{equation*}
	\bm{v}' = \bm{A}' + \bm{B}'x^{m} \log x + \bm{C}'x^{m} (\log x)^{2}
\end{equation*}
satisfies
\begin{equation*}
	\Delta_{+} \bm{v}' = \| d(\Dot{\mathcal{P}}_{\theta}(m)1 ) \|^{2}_{g_{+}} -2 + O(x^{\infty}).
\end{equation*}
The function $\bm{A}'$ is unique modulo $O(x^{m})$,
and $\bm{B}'$ and $\bm{C}'$ are unique modulo $O(x^{\infty})$.
Moreover
\begin{equation} \label{eq:coeff_scat_qprime2}
	\bm{B}'|_{M} = -2c_{m} \bm{Q}'_{\theta}, \qquad \bm{C}'|_{M} = 4c_{m}Q_{\theta}.
\end{equation}
\end{prop}

Next we recall definitions of the $P$-prime operator
and $Q$-prime curvature defined in~\cite{Hirachi14}.

\begin{defn} \label{def:p_prime}
Let $f$ be a CR pluriharmonic function on $M$
and $\widetilde{f}$ a smooth extension of $f$ to $\overline{\Omega}$
that is pluriharmonic near $M$.
There are $F', G' \in C^{\infty}(\overline{\Omega})$ such that $F'|_{M} = 0$ and
\begin{equation*}
	\Delta_{+}(\widetilde{f} \log x + F' + G'x^{m} \log x) = m\widetilde{f} + O(x^{\infty}).
\end{equation*}
The function $F'$ is unique modulo $O(x^{m})$,
and $G'$ is unique modulo $O(x^{\infty})$.
The \emph{$P$-prime operator} $P'_{\theta}$ is an operator
from the space of CR pluriharmonic functions to $C^{\infty}(M)$ defined by
\begin{equation} \label{eq:coeff_pprime}
	G'|_{M} = -2c_{m} P'_{\theta} f.
\end{equation}
In particular, $P'_{\theta}1 = Q_{\theta}$ holds.
\end{defn}

\begin{defn} \label{def:q_prime}
Let $\theta$ be a pseudo-Einstein contact form on $M$ and
$x$ a Fefferman's defining function normalized by $\theta$.
There are smooth functions $A'$ and $B'$ on $\overline{\Omega}$
such that $A'|_{M} = 0$ and
\begin{equation*}
	\Delta_{+}(A' + B' x^{m} \log x) = \| d \log x \|^{2}_{g_{+}} -2 + O(x^{\infty}).
\end{equation*}
The function $A'$ is unique modulo $O(x^{m})$,
and $B'$ is unique modulo $O(x^{\infty})$.
The \emph{$Q$-prime curvature} $Q'_{\theta}$ is a function on $M$ defined by
\begin{equation} \label{eq:coeff_qprime}
	B'|_{M} = -2c_{m}Q'_{\theta}.
\end{equation}
\end{defn}

\begin{proof}[Proof of Theorem~\ref{thm:coincidence_p_prime}]

Let $f$ be a CR pluriharmonic function
and $\widetilde{f}$ a smooth extension of $f$ that is pluriharmonic near the boundary.
Then $\Delta_{+} \widetilde{f}$ has compact support in $\Omega$
since $g_{+}$ is K\"ahler near $M$.
From Theorem~\ref{thm:resolvent},
$K \widetilde{f} = c_{m}^{-1} x^{-m} \mathcal{R}(m)\Delta_{+} \widetilde{f}$ is smooth up to the boundary.
Hence $\mathcal{P}_{\theta}(m)f = \widetilde{f} - c_{m} (K\widetilde{f}) x^{m}$ holds.
Using Proposition~\ref{prop:formal_scat_p_prime2}
and Definition~\ref{def:p_prime},
we have modulo $O(x^{\infty})$,
\begin{align*}
	mc_{m} (K\widetilde{f}) x^{m}
	&= m (\widetilde{f} - \mathcal{P}_{\theta}(m)f) \\
	&= \Delta_{+} [(\widetilde{f} - \mathcal{P}_{\theta}(m)f) \log x + (F' - \bm{F}') \\
	& \mspace{54mu} + (G'-\bm{G}') x^{m} \log x + \bm{H}'x^{m} (\log x)^{2}]  \\
	&= \Delta_{+} [(F'\!-\!\bm{F}') \!+\! (G'\!-\!\bm{G}'\!+\! c_{m} K\widetilde{f})x^{m} \log x
	\!+\! \bm{H}'x^{m} (\log x)^{2}]. 
\end{align*}
Lemma~\ref{lem:laplacian} gives
\begin{gather*}
	F' - \bm{F}' \in x^{m}C^{\infty}(\overline{\Omega}), \\
	- m(G'|_{M} - \bm{G}'|_{M} + c_{m} (K\widetilde{f})|_{M} )  -2 \bm{H}'|_{M}
	= mc_{m}(K \widetilde{f} )|_{M}, \\
\intertext{and}
	-m\bm{H}'|_{M} = 0.
\end{gather*}
Hence we obtain
\begin{equation*}
	\bm{P}'_{\theta}f = P'_{\theta}f - (K \widetilde{f} )|_{M}
\end{equation*}
from~\eqref{eq:coeff_scat_pprime2} and~\eqref{eq:coeff_pprime}.
Assume that the metric $g_{+}$ is K\"ahler and
every CR pluriharmonic function has a pluriharmonic extension
to $\Omega$.
Then
we can choose as $\widetilde{f}$ its pluriharmonic extension.
Since $g_{+}$ is K\"ahler,
$\widetilde{f}$ satisfies $\Delta_{+} \widetilde{f} = 0$.
Therefore $K\widetilde{f} = 0$ and $\bm{P}'_{\theta}f = P'_{\theta}f$.
\end{proof}

\begin{proof}[Proof of Theorem~\ref{thm:coincidence_q_prime}]
Let $\theta$ be a pseudo-Einstein contact form on $M$
and $x$ a Fefferman's defining function normalized by $\theta$.
Since $\Pi$ has compact support in $\Omega$,
so does $\tr_{g_{+}} \Pi = m - \Delta_{+} \log x$.
Hence the function
\begin{equation} \label{eq:q_prime_correction}
	D = c_{m}^{-1} x^{-m} \mathcal{R}(m) (\Delta_{+} \log x - m)
\end{equation}
is smooth up to the boundary and
\begin{equation*}
	\Dot{\mathcal{P}}_{\theta}(m)1 = - \log x + c_{m} D x^{m}.
\end{equation*}
Therefore
\begin{equation*}
	\| d (\Dot{\mathcal{P}}_{\theta}(m)1) \|_{g_{+}}^{2} -2
	= \| d \log x \|^{2}_{g_{+}} -2 -4 m c_{m} D x^{m} + O(x^{m+1}).
\end{equation*}
From~\eqref{eq:coeff_scat_qprime2} and~\eqref{eq:coeff_qprime},
we have
\begin{equation*}
	2mc_{m} \bm{Q}'_{\theta}
	= 2mc_{m}Q'_{\theta} - 4mc_{m}D|_{M}.
\end{equation*}
In particular if $\Pi = 0$,
then $D = 0$ and $\bm{Q}'_{\theta} = Q'_{\theta}$ holds.
\end{proof}

We need to prepare some notation to prove Theorem~\ref{thm:renormalized_volume}.
Take a $(1,0)$-vector field $\xi$ on $\mathcal{N}$ 
and a smooth function $\kappa$ such that
\begin{equation*}
	\xi x = -1, \quad \iota_{\xi} \partial \overline{\partial} x = \kappa \overline{\partial} x
	\quad \text{near} \ M,
\end{equation*}
and write $\xi = N - (\sqrt{-1}/2) T$.
For the $1$-form $\vartheta = (\sqrt{-1}/2)(\partial - \overline{\partial})x$,
$\vartheta(N) = 0$ and $\vartheta(T) = 1$ hold.
Let $\widetilde{T}^{1,0}$ be the complex subbundle of $T^{1,0}X$ defined by
\begin{equation*}
	\widetilde{T}^{1,0} = \{ W \in T^{1,0}X \mid Wx = 0 \},
\end{equation*}
and $(W_{1}, \dots , W_{n})$ a local frame of $\widetilde{T}^{1,0}$.
Then $(\xi, W_{1}, \dots , W_{n})$ is a local frame of $T^{1,0}X$
and the dual $(1,0)$-coframe is of the form
$(-\partial x, \vartheta^{1}, \dots , \vartheta^{n})$.
Using this coframe,
we can write
\begin{equation*}
	d\vartheta = - \sqrt{-1} \partial \overline{\partial} x 
	= \sqrt{-1} \sum_{\alpha, \beta = 1}^{n} h_{\alpha \overline{\beta}} 
	\vartheta^{\alpha} \wedge \vartheta^{\overline{\beta}}
	- \kappa dx \wedge \vartheta,
\end{equation*}
where $(h_{\alpha \overline{\beta}})$ is a positive definite Hermitian matrix
from strict pseudoconvexity.

Since the real vector field $-N$ is transversal to the boundary and inward pointing,
we obtain a smooth map $\Psi \colon [0, \delta) \times M \to \overline{\Omega}$
such that $\Psi$ is an orientation-reversing diffeomorphism onto the image,
and $\Psi_{*} (\partial / \partial r) = -N$.
Here $r$ is a coordinate of $[0,\delta)$.
We extend $\theta$ constantly in the $r$-direction and
write this $1$-form on $[0,\delta) \times M$ as $\theta$ also.
There exists a unique smooth function $s$ on $[0,\delta) \times M$ such that
\begin{equation*}
	\Psi^{*} (\vartheta \wedge (d \vartheta)^{n} ) 
	\equiv s \theta \wedge (d\theta)^{n} \quad \mod \  dr.
\end{equation*}
Contracting with the vector field $\partial / \partial r$,
we have that the term containing $dr$ is equal to zero,
that is,
\begin{equation*}
	\Psi^{*} (\vartheta \wedge (d \vartheta)^{n} ) = s \theta \wedge (d\theta)^{n}.
\end{equation*}

Let $\widetilde{\xi} = \widetilde{N} - (\sqrt{-1}/2)\widetilde{T}$ be a $(1,0)$-vector field such that
\begin{equation*}
	\widetilde{\xi} x = -1, \quad \widetilde{\xi} \perp_{g_{+}} \ker \partial x.
\end{equation*}
Then $\widetilde{\xi} = \xi - x \sum_{\alpha} \eta^{\alpha}W_{\alpha}$
for some smooth functions $\eta^{\alpha}$,
and the dual frame of $(\widetilde{\xi}, W_{\alpha})$
is given by
$(\widetilde{\vartheta}^{0} = - \partial x,
\widetilde{\vartheta}^{\alpha} = \vartheta^{\alpha} + x \eta^{\alpha} \partial x)$.
The matrix representation of the metric $g_{+}$ for this frame is
\begin{equation*}
	\begin{pmatrix}
	\frac{1+ax}{x^{2}} & 0 \\
	0 & \frac{\widetilde{h}_{\alpha \overline{\beta}}}{x}
	\end{pmatrix},
\end{equation*}
where $a$ is a smooth function 
and $(\widetilde{h}_{\alpha \overline{\beta}})$ is a positive definite Hermitian matrix
satisfying $\widetilde{h}_{\alpha \overline{\beta}}|_{M} = h_{\alpha \overline{\beta}}|_{M}$.
The volume form $\dvol_{g_{+}}$ for $g_{+}$ is, near $M$,
written as
\begin{align*}
	\dvol_{g_{+}}
	&= - \frac{(\sqrt{-1})^{n}}{n!} \frac{1+ax}{x^{n+2}} dx \wedge \vartheta 
	\wedge \left(\sum_{\alpha, \beta = 1}^{n}
	\widetilde{h}_{\alpha \overline{\beta}} \widetilde{\vartheta}^{\alpha} \wedge
	\widetilde{\vartheta}^{\overline{\beta}} \right)^{n} \\
	&= - \frac{1}{n!} \frac{1+ax}{x^{n+2}} 
	\frac{\det \widetilde{h}_{\alpha \overline{\beta}}}{\det h_{\alpha \overline{\beta}}}
	dx \wedge \vartheta \wedge (d \vartheta)^{n}.
\end{align*}
Define $\psi^{(k)} \in C^{\infty} (M)$ by
\begin{equation*}
	\Psi^{*}\left( \frac{\det \widetilde{h}_{\alpha \overline{\beta}}}
	{\det h_{\alpha \overline{\beta}}} \right) s
	= \sum_{k=0}^{\infty} \psi^{(k)} t^{k}, \quad \psi^{(0)} = 1.
\end{equation*}

\begin{proof}[Proof of Theorem~\ref{thm:renormalized_volume}]

Since $\tr_{g_{+}} \Pi = m - \Delta_{+} \log x = O(x^{m})$,
we have
\begin{equation*}
	- \Dot{\mathcal{P}}_{\theta}(m) 1
	= \log x + Ax^{m} + Bx^{m} \log x,
\end{equation*}
where $A,B \in C^{\infty}(\overline{\Omega})$
with $B|_{M} = -2c_{m} Q_{\theta}$
from the proof of Proposition~\ref{prop:formal_scat_p_prime}.
Since the outward unit normal vector field of $\partial \Omega_{\varepsilon}$ is
written as $\sqrt{2}(1+a\varepsilon)^{-1/2} \varepsilon \widetilde{N}$,
\begin{align*}
	\int_{\Omega_{\varepsilon}} \dvol_{g_{+}}
	&= \frac{1}{m} \int_{\Omega_{\varepsilon}} \Delta_{+}
	(-\Dot{\mathcal{P}}_{\theta}(m)1) \dvol_{g_{+}} \\
	&= - \frac{1}{m} \int_{\partial \Omega_{\varepsilon}}
	\widetilde{N}(-\Dot{\mathcal{P}}_{\theta}(m)1) \frac{(\sqrt{-1})^{n}}{n!}
	\frac{1}{\varepsilon^{n}} \vartheta 
	\wedge (\widetilde{h}_{\alpha \overline{\beta}} \widetilde{\vartheta}^{\alpha} \wedge
	\widetilde{\vartheta}^{\overline{\beta}} )^{n} \\
	&= - \frac{1}{m} \int_{M}
	\Psi^{*}(\widetilde{N}(-\Dot{\mathcal{P}}_{\theta}(m)1)) \frac{1}{n!} \frac{1}{\varepsilon^{n}}
	\Psi^{*} \left( \frac{\det \widetilde{h}_{\alpha \overline{\beta}}}
	{\det h_{\alpha \overline{\beta}}} \right)
	s \, \theta \wedge (d\theta)^{n} \\
	&= \frac{1}{m!} \int_{M}
	\left( \frac{1}{\varepsilon^{m}} + mB|_{M}  \log \varepsilon 
	+m A|_{M}  + B|_{M}  \right) \\
	& \mspace{90mu} \times \Psi^{*}\left( \frac{\det \widetilde{h}_{\alpha \overline{\beta}}}
	{\det h_{\alpha \overline{\beta}}} \right)
	s \, \theta \wedge (d\theta)^{n} + o(1).
\end{align*}
Hence
if we write
\begin{equation*}
	\int_{\Omega_{\varepsilon}} \dvol_{g_{+}}
	= \sum_{j=1}^{m} \frac{b_{j}}{\varepsilon^{m+1-j}} + L \log \varepsilon + V + o(1),
\end{equation*}
then
\begin{align}
	L
	&=  \frac{1}{n!} \int_{M} B|_{M} \, \theta \wedge (d\theta)^{n} 
	= - \frac{2c_{m}}{n!} \int_{M} Q_{\theta} \, \theta \wedge (d\theta)^{n} \label{eq:logpart} \\
\intertext{and}
	m! V
	&= \int_{M} \psi^{(m)} \theta\wedge (d\theta)^{n}
	 + \int_{M} \left[ mA|_{M} - 2c_{m}Q_{\theta} \right] \theta \wedge (d\theta)^{n}. \label{eq:finpart}
\end{align}
On the other hand, we have
\begin{align*}
	&\quad\ -n! \int_{\Omega_{\varepsilon}}
	\| d  (\Dot{\mathcal{P}}_{\theta}(m)1) \|^{2}_{g_{+}} \dvol_{g_{+}}   \\
	& = - \int_{[\varepsilon, \delta) \times M}
	\Psi^{*} \left( \| d  (\Dot{\mathcal{P}}_{\theta}(m)1) \|^{2}_{g_{+}}
	\frac{1+ax}{x^{n+2}} 
	\frac{\det \widetilde{h}_{\alpha \overline{\beta}}}{\det h_{\alpha \overline{\beta}}}
	dx \wedge \vartheta \wedge (d \vartheta)^{n} \right) + O(1)  \\
	& = - 2 \int_{[\varepsilon, \delta) \times M}
	\Biggl( \frac{1}{t^{m+1}} \Psi^{*}
	\left( \frac{\det \widetilde{h}_{\alpha \overline{\beta}}}{\det h_{\alpha \overline{\beta}}} \right) s
	+2mB|_{M} \frac{\log t}{t} \\
	&\mspace{120mu}+ 2(mA|_{M} + B|_{M})\frac{1}{t} \Biggr)
	dt \wedge \theta \wedge (d\theta)^{n} + O(1).
\end{align*}
The coefficient of $\log \varepsilon$ in the formula above is
\begin{equation*}
	2 \int_{M} \psi^{(m)} \theta \wedge (d\theta)^{n}
	+ 4\int_{M} [mA|_{M} - 2c_{m} Q_{\theta}]  \theta \wedge (d\theta)^{n}.
\end{equation*}
Hence from~\eqref{eq:logpart} and~\eqref{eq:finpart},
the quantity $m!V$ is equal to the sum of
\begin{equation*}
	-\int_{M} [ mA|_{M} - 4c_{m} Q_{\theta}] \theta \wedge (d\theta)^{n}
\end{equation*}
and the coefficient of $\log \varepsilon$ in
\begin{equation} \label{eq:volume}
	- \frac{n!}{2} \int_{\Omega_{\varepsilon}}
	[\| d  \Dot{\mathcal{P}}_{\theta}(m)1 \|^{2}_{g_{+}} -2] \dvol_{g_{+}}.
\end{equation}
From Proposition~\ref{prop:formal_scat_q_prime2},
the coefficient of $\log \varepsilon$ in~\eqref{eq:volume} is equal to
\begin{align*}
	&\quad -\frac{1}{2} m \int_{M} \bm{B}'|_{M} \, \theta \wedge (d\theta)^{n}
	- \int_{M} \bm{C}'|_{M} \, \theta \wedge (d\theta)^{n}  \\
	& = mc_{m} \int_{M} \bm{Q}'_{\theta}  \, \theta \wedge (d\theta)^{n}
	-4c_{m} \int_{M} Q_{\theta}  \, \theta \wedge (d\theta)^{n}.
\end{align*}
Therefore we have~\eqref{eq:renormalized_volume}.
\end{proof}

\begin{proof}[Proof of Theorem~\ref{thm:renormalized_volume2}]
Assume that 
$g_{+}$ is K\"ahler and
$\theta$ is pseudo-Einstein.
Let $x$ be a Fefferman's defining function normalized by $\theta$
and $D \in C^{\infty}(\overline{\Omega})$ as in~\eqref{eq:q_prime_correction}.
Then $\bm{Q}'_{\theta} = Q'_{\theta} -2D|_{M}$ holds.
Hence the integral of $\bm{Q}'_{\theta}$ is written as
\begin{equation*}
	\int_{M} \bm{Q}'_{\theta} \, \theta \wedge (d\theta)^{n}
	= \int_{M} Q'_{\theta} \, \theta \wedge (d\theta)^{n}
	- 2 \int_{M} D|_{M} \, \theta \wedge (d\theta)^{n}.
\end{equation*}
On the other hand,
we have
\begin{equation*}
	(\tr_{g_{+}} \Pi) \cdot \omega_{+}^{m}
	= m \Pi \wedge \omega_{+}^{n}
	= m \Pi^{m} + d \tau
\end{equation*}
for a compactly supported $(2m-1)$-form $\tau$
and
\begin{equation*}
	\int_{\Omega} (\tr_{g_{+}} \Pi) \cdot \omega_{+}^{m}
	= m \int_{\Omega} \Pi^{m}.
\end{equation*}
The left hand side is equal to
\begin{equation*}
	- c_{m} \int_{\Omega} \Delta_{+} (Dx^{m}) \omega_{+}^{m}
	= - m^{2} c_{m} \int_{M} D|_{M} \, \theta \wedge (d\theta)^{n}.
\end{equation*}
Thus we obtain
\begin{equation*}
	\int_{M} \bm{Q}'_{\theta} \, \theta \wedge (d\theta)^{n}
	= \int_{M} Q'_{\theta} \, \theta \wedge (d\theta)^{n}
	+ \frac{2}{mc_{m}} \int_{\Omega} \Pi^{m}.
\end{equation*}
Since $- \Dot{\mathcal{P}}_{\theta}(m)1$ is equal to $\log x -c_{m} Dx^{m}$,
we have $A = -c_{m} D$ and $B = 0$.
Hence
\begin{equation*}
	\int_{M} A|_{M} \, \theta \wedge (d\theta)^{n}
	= \frac{1}{mc_{m}} \int_{\Omega} \Pi^{m}.
\end{equation*}
The last statement follows from~\eqref{eq:renormalized_volume},~\eqref{eq:integral_scat_qprime},
and~\eqref{eq:integral_correction}.
\end{proof}

\section*{Acknowledgements}
The author is grateful to his supervisor Professor Kengo Hirachi
for introducing him to this subject and for helpful comments.
This work was supported by the Program for Leading Graduate Schools, MEXT, Japan.

\end{document}